\documentclass{amsart}

\usepackage{amssymb}
\usepackage{amsmath}
\usepackage{amsthm}
\usepackage{amsbsy}
\usepackage{amscd}
\usepackage{bm}

\theoremstyle{plain}
\newtheorem{theorem}{Theorem}[section]
\newtheorem{lemma}[theorem]{Lemma}

\newtheorem{proposition}[theorem]{Proposition}

\theoremstyle{definition}
\newtheorem{definition}[theorem]{Definition}

\theoremstyle{remark}
\newtheorem{remark}[theorem]{Remark}

\newcommand{\R}{\mathbb{R}}

\newcommand{\M}{\widetilde{M}}

\newcommand{\incl}{\hookrightarrow}

\begin{document}

\title{Algebraic and Geometric intersection numbers for free groups}

\author{Siddhartha Gadgil}

\address{   Department of Mathematics,\\
        Indian Institute of Science,\\
        Bangalore 560003, India}

\email{gadgil@isibang.ac.in}

\author{Suhas Pandit}
\address{  Stat-math Unit,\\
          Indian Statistical Institute,\\
       Bangalore 560059, India}
\email{suhas@isibang.ac.in}

\date{\today}


\begin{abstract}
We show that the algebraic intersection number of Scott and Swarup for
splittings of free groups coincides with the geometric intersection
number for the sphere complex of the connected sum of copies of
$S^2\times S^1$.

\end{abstract}

\maketitle

\section{Introduction}

The \emph{geometric intersection number} of homotopy classes of
(simple) closed curves on a surface is the minimum number of
intersection points of curves in the homotopy classes. This is a
much studied concept and has proved to be extremely useful in
low-dimensional topology.

Scott and Swarup~\cite{SS} introduced an algebraic analogue, called
the \emph{algebraic intersection number}, for a pair of
\emph{splittings} of groups. This is based on the associated partition
of the ends of a group~\cite{St}. Splittings of groups are the natural
analogue of simple closed curves on a surface $F$ -- splittings of
$\pi_1(F)$ corresponding to homotopy classes of simple closed curves
on $F$. Scott and Swarup showed that, in the case of surfaces, the
algebraic and geometric intersection numbers coincide.

We show here that the analogous result holds for free groups,
viewed as the fundamental group of the connected sum $ M =
\sharp_n S^2 \times S^1 $ of $n$ copies of $S^2 \times S^1$.
Observe that this is a closed $3$-manifold with fundamental group
the free group on $n$ generators. Thus, the manifold can be
regarded as a model for studying the free group and its
automorphisms.

Embedded spheres in $M$ correspond to splittings of the free
group. Hence, given a pair of embedded spheres in $M$, we can
consider their \emph{geometric intersection number} (defined
below) as well as the \emph{algebraic intersection number} of
Scott and Swarup for the corresponding splittings. Our main
result is that, for embedded spheres in $M$ these two
intersection numbers coincide. The principal method we use is the
normal form for embedded spheres developed by Hatcher.

Before stating our result, we recall the definition of the
intersection numbers.

\begin{definition}
Let $A$ and $B$ be two isotopy classes of embedded spheres $S$ and
$T$, respectively, in $M$. The \emph{geometric intersection number}
$I(A,B)$ of $A$ and $B$ is defined as the minimum of the number of
components $| S\cap T |$ of $S\cap T$ over embedded transversal
spheres $S$ and $T$ representing the isotopy classes $A$ and $B$,
respectively.
\end{definition}

This is clearly symmetric. Further, for an embedded sphere $S$, if
$A=[S]$ then $I(A,A)=0$.

We consider next the {algebraic intersection number}. Let
$\widetilde{M}$ be the universal cover of $M$. Observe that
$\pi_2(M)=\pi_2(\M)=H_2(\M)$. The fundamental group $\pi_1(M)=G$ of
$M$, which is a free group of rank $n$, acts freely on the universal
cover $\widetilde{M}$ of $M$ by deck transformations.  Homotopy
classes of spheres in $M$ correspond to equivalence classes of
elements in $H_2(\M)$ up to the action of deck transformations. For
embedded spheres, we can consider isotopy classes instead of homotopy
classes as the homotopy classes of embedded spheres are the same as
isotopy classes of embedded spheres~\cite{La}.

For an embedded sphere $S\in M$ with lift $\widetilde{S}\in
\widetilde{M}$, all the translates of $\widetilde{S}$ are embedded and
disjoint from $\widetilde{S} $. In particular, if
$\widetilde{A}=[\widetilde{S}]$ is the isotopy class represented by
$\widetilde{S}$, then $\widetilde{A}$ and $g\widetilde{A}$ can be
represented by disjoint embedded spheres for each deck transformation
$g \in G$.

\begin{definition}

Let $A = [S]$ and $B = [T]$ be two isotopy classes of embedded spheres
$S$ and $T$, respectively, in $M$. Let $\widetilde{A} = [\widetilde{S}]$
and $\widetilde{B} = [\widetilde{T}]$, where $\widetilde{S}$ and
$\widetilde{T}$ are the lifts of $S$ and $T$, respectively, to
$\widetilde{M}$. The \emph{algebraic intersection number}
$\widetilde{I}(A,B)$ of $A$ and $B$ is defined as the number of
translates $g\widetilde{B}$ of $\widetilde{B}$ such that
$\widetilde{A}$ and $g\widetilde{B}$ can not be represented by
disjoint embedded spheres in $\widetilde{M}$.
\end{definition}
It was shown in ~\cite{Ga} that this coincides with the algebraic
intersection number of Scott and Swarup.

We say that two isotopy classes $\widetilde{A}= [\widetilde{S}]$ and
$\widetilde{B} = [\widetilde{T}]$ of embedded spheres in
$\widetilde{M}$ \emph{cross} if they cannot be represented by disjoint
embedded spheres. Thus, the algebraic intersection number is the
number of elements $g\in \pi_1(M)$ such that $\widetilde{A}$ and
$g\widetilde{B}$ cross. We shall also say that $\widetilde{S}$ and
$\widetilde{T}$ cross if the classes they represent cross.

It is immediate that $\widetilde{A}$ and $g\widetilde{B}$ cross if and
only if $g^{-1}\widetilde{A}$ and $\widetilde{B}$ cross. It follows
that $\tilde{I}(A,B)=\tilde{I}(B,A)$. Thus the \emph{algebraic
intersection number} is symmetric.

Clearly, for all but finitely many translates $g\widetilde{B}$ of
$\widetilde{B}$, $\widetilde{A}$ and $g\widetilde{B}$ can be
represented by disjoint embedded spheres in $\widetilde{M}$. This
is because, for any pair of embedded spheres $S$ and $T$ in $M$,
all but finitely many translates of $\widetilde{T}$ are disjoint
from $\widetilde{S}$ in $ \widetilde{M}$. Hence
$\widetilde{I}(A,B)$ is finite for all isotopy classes $A$ and
$B$ of embedded spheres in $M$.

As was shown in~\cite{Ga}, it follows from results of Scott and Swarup
that if the algebraic intersection number between classes $A$ and $B$
as above vanishes, then they can be represented by disjoint embedded
spheres, i.e., their geometric intersection number vanishes. The
converse is an easy observation.

We prove here a much stronger result -- that the algebraic and
geometric intersection numbers are equal.

\begin{theorem}\label{T:main}
For isotopy classes $A$ and $B$ of embedded spheres in $M$,
$\widetilde{I}(A,B)=I(A,B)$.
\end{theorem}

Our proof is based on the normal form for spheres in $M$ due to
Hatcher~\cite{Ht}, which we recall in Section~\ref{S:norm}. We extend a
sphere $\Sigma$ in the isotopy class $B$ to a maximal system of
spheres and consider a sphere $S$ in the isotopy class of $A$ in
normal form with respect to this system. We then show in
Section~\ref{S:pf} that, when $S$ is in normal form, the number
of components of intersection between $S$ and $\Sigma$ is the
algebraic intersection number between the isotopy classes $ A =
[S]$ and $B$.

Our methods also show that, if $A_1$,\dots $A_n$ is a collection of
isotopy classes of embedded spheres, each pair of which can be
represented by disjoint spheres, then all the classes $A_i$ can be
simultaneously represented by disjoint spheres. We prove this in
Theorem~\ref{T:disj}.

The sphere complex associated to $M$ is a simplicial complex whose
vertices are the isotopy classes of embedded spheres in $M$. A
set of isotopy classes of embedded spheres in $M$ is deemed to
span a simplex if they can be realized disjointly in $M$.  This
is an analogue of the curve complex associated to a surface. The
topological properties of the sphere complex have been studied by
Hatcher, Hatcher-Vogtmann and Hatcher-Wahl in~\cite{Ht},
\cite{HV1}, \cite{HV2}, \cite{HV3}, \cite{HV4}, \cite{HW}. Culler
and Vogtmann have constructed a contractible complex
\textit{Outer space} on which the outer automorphism group
$Out(F_n)$ of the free group $F_n$ acts discretely and with
finite stabilizers~\cite{CV}.  This is an analogue of Teichm\"
uller space of surface on which the mapping class group of the
surface acts. Culler and Morgan have constructed a
compactification of Outer space much like Thurston's
compactification of Teichm\" uller space~\cite{CM}. The curve
complex has proved to be fruitful in studying Teichum\"uller
space (see, for instance, \cite{Iv1}, \cite{Iv2}).

The \emph{geometric intersection number} of curves on a surface has
been used to give constructions like the space of measured laminations
whose projectivization is the boundary of Teichm\" uller space,
~\cite{Lu}, as well as to study geometric properties, including
hyperbolicity of the curve complex in ~\cite{Bo},~\cite{MM}. One may
hope that the \emph{geometric intersection number } of embedded
spheres in $M$ might be useful to give such constructions in case
sphere complex and Outer space.  The sphere complex is useful for
studying the mapping class group of $M$, $Out(F_n)$ and Outer space.

An important ingredient of our proofs is the observation that if $S$
and $T$ are embedded spheres in $M$ and $S$ is in normal form with
respect to a maximal system of spheres containing $T$, then $S$ and
$T$ intersect minimally. This is somewhat analogous to results for
geodesics and least-area surfaces~\cite{FHS1}\cite{FHS2}. Further the
components of intersection correpsond to \emph{crossing}. This is very
similar to the case of geodesics, where intersections correspond to
linking of end points.

\section{Normal spheres}\label{S:norm}

We recall the notion of \emph{normal sphere systems} from~\cite{Ht}.

\begin{definition}
A smooth, embedded $2$-sphere in $M$ is said to be essential if it
does not bound a $3$-ball in $M$.
\end{definition}

\begin{definition}
A \emph{system of $2$-spheres} in $M$ is defined as a finite
collection of disjointly embedded, pair-wise non-isotopic, essential
smooth $2$-spheres $S_i \subset M$.
\end{definition}

Let $\Sigma = \cup_j \Sigma_j$ be a maximal system of $2$-sphere
in $M$. Splitting $M$ along $\Sigma$, then produces a finite
collection of $3$-punctured $3$-spheres $P_k$. Here a
$3$-punctured $3$-sphere is the complement of the interiors of
three disjointly embedded $3$-balls in a $3$-sphere.

\begin{definition}
 A system of $2$-spheres $S = \cup_i S_i$ in $M$ is said to be in
\emph{normal form with respect to $\Sigma$ } if each $S_i$ either
coincides with a sphere $\Sigma_j$ or meets $\Sigma$ transversely
in a non empty finite collection of circles splitting $S_i$ into
components called pieces, such that the following two conditions
hold in each $P_k$:
\begin{enumerate}
\item
Each piece in $P_k$ meets each component of $\partial P_k$ in at most
one circle.
\item
No piece in $P_k$ is a disk which is isotopic, fixing its boundary, to a
disk in $\partial P_k$.
\end{enumerate}
\end{definition}

Thus each piece is a disk, a cylinder or a pair of pants. A disk piece
has its boundary on one component of $\partial P_k$ and separates the
other two components of $\partial P_K$.

Recall the following result from~\cite{Ht}.

\begin{proposition}[Hatcher]
Every system $S \subset M$ can be isotoped to be in normal form with
respect to $\Sigma$. In particular, every embedded sphere $S$ which
does not bound a ball in $M$ can be isotoped to be in normal form with
respect to $\Sigma$.
\end{proposition}

We recall some constructions from~\cite{Ht}. First, we
associate a tree $T$ to $\widetilde{M}$ corresponding to the
decomposition of $M$ by $\Sigma$.  Let $\widetilde{\Sigma}$ be the
pre-image of $\Sigma$ in $\widetilde{M}$. The closure of each
component of $\M-\widetilde\Sigma$ is a $3$-punctured $3$-sphere. The
vertices of the tree are of two types, with one vertex corresponding
to the closure of each component of $\M-\widetilde\Sigma$ and one
vertex for each component of $\widetilde\Sigma$. An edge of
$T$ joins a pair of vertices if one of the vertices corresponds to the
closure of a component $X$ of $\M-\widetilde\Sigma$ and the other
vertex corresponds to a component of $\widetilde\Sigma$ that is in the
boundary of $X$. Thus, we have a $Y$-shaped subtree corresponding to
each complementary component. We pick an embedding of $T$ in $\M$
respecting the correspondences.

Given a sphere $S$ in normal form with respect to $\Sigma$ and a lift
$\widetilde{S}$ of $S$ to $\widetilde{M}$, we associate a tree
$T(\widetilde{S})$ corresponding to the decomposition of
$\widetilde{S}$ into pieces.  The tree has two types of vertices,
vertices corresponding to closures of components of $\widetilde{S} -
\widetilde{\Sigma}$ (i.e., pieces) and vertices corresponding to each
component of $\widetilde{S} \cap \widetilde{\Sigma}$. Edges join a
pair of vertices if one of the vertices corresponds to a piece and the
other to a boundary component of the piece.

In~\cite{Ht}, it is shown that $T(\widetilde{S})$ is a tree. Moreover,
the inclusion $\widetilde{S} \incl \widetilde{M}$ induces a natural
inclusion map $T( \widetilde{S}) \incl T$. So we can view
$T(\widetilde{S})$ as a subtree of $T$. The bivalent vertices of $T$
correspond to spheres components in $\widetilde{\Sigma}$, i.e., lifts
of the spheres $\Sigma_j$ and their translates.

\section{Algebraic and Geometric Intersection numbers}\label{S:pf}

Consider now two isotopy classes $A$ and $B$ of embedded spheres in
$M$ . Choose an embedded sphere $\Sigma_1$ in the isotopy class $B$
and extend this to a maximal collection $\Sigma$ of spheres.  Let $S$
be a representative for $A$ in normal form with respect to
$\Sigma$. Theorem~\ref{T:main} is equivalent to showing that
$\widetilde{I}(A,[\Sigma_j])= I(A, [\Sigma_j])$ for $j=1$. We begin by
showing the non-trivial inequality here.

\begin{lemma}\label{L:ineq}
 If $A =[S]$ is the isotopy class of the embedded sphere $S$ in $M$, then for
 the isotopy class $[\Sigma_j]$ of $\Sigma_j$ in $M$,
 $\widetilde{I}(A,[\Sigma_j])\geq I(A, [\Sigma_j])$.
\end{lemma}
\begin{proof}
The sphere $S$, which is in normal form with respect to $\Sigma$,
represents the class $A$. We shall show that the number of components
of intersection of $S$ with $\Sigma_j$ is
$\widetilde{I}(A,[\Sigma_j])$. As the geometric intersection number is
the minimum of the number of components of intersection of spheres in
the isotopy classes, the lemma is an immediate consequence of this
claim.

Fix a lift $\widetilde{S}$ of $S$. The components of
$S\cap\Sigma_j$ are homotopically trivial circles in $M$. These
lift to circles of intersection between $\widetilde S$ and
components of the pre-image of $\Sigma_j$. These correspond to
vertices of $T(\widetilde{S})$. As $T(\widetilde{S})$ is a tree
which embeds in $T$, different circles of intersection of $S$ and
$\Sigma_j$ correspond to intersections of $\tilde{S}$ with
different components of the pre-image of $\Sigma_j$. It follows
that the number of components of intersection of $S$ with
$\Sigma_j$ is the number of components of the pre-image of
$\Sigma_j$ that intersect $\widetilde{S}$.

The main observation needed is the following lemma.
\begin{lemma}\label{L:cross}
If $\widetilde{S}$ intersects a component $\widetilde{\Sigma_j}$ of
the pre-image of $\Sigma_j$, then the spheres $\widetilde{S}$ and
$\widetilde{\Sigma_j}$ cross.
\end{lemma}
\begin{proof}
Assume that $\widetilde{S}$ intersects the component
$\widetilde{\Sigma_j}$ of the pre-image of $\Sigma_j$. The sphere
$\widetilde{\Sigma_j}$ corresponds to a vertex $v_0$ of $T$. As
$\widetilde{S}$ intersects $\widetilde{\Sigma_j}$ and $S$ is in normal
form, the vertex $v_0$ is an interior vertex of $T(\widetilde{S})$.

We recall the notion of crossing due to Scott and Swarup, which
by~\cite{Ga} is equivalent to the notion we use. The spheres
$\widetilde{S}$ and $\widetilde{\Sigma_j}$ partition the ends of $\M$
into pairs of complementary subsets $E_S^{\pm}$ and
$E_{\Sigma}^{\pm}$, corresponding to the components of the complement
of the respective spheres in $\widetilde{M}$. The spheres
$\widetilde{S}$ and $\widetilde{\Sigma_j}$ cross if all the four
intersections $E_S^{\pm}\cap E_{\Sigma}^{\pm}$ are non-empty.

A properly embedded path $c:\R \to \M$ induces a map from the ends
$\pm\infty$ of $\R$ to the ends of $\M$. Thus, we can associate to $c$
a pair of ends $c_{\pm}$. We say that the path $c$ is a path from
$c_-$ to $c_+$. Poincar\'e duality gives a useful criterion for when
two ends $E$ and $E'$ of $\M$ are in different equivalence classes with
respect to the partition corresponding to $\widetilde{S}$. Namely, $E$
and $E'$ are in different equivalence classes if and only if there is
a proper path $c$ from $E$ to $E'$ so that $c\cdot \widetilde{S}=\pm
1$, with $c\cdot \widetilde{S}$ the intersection pairing obtained from
the cup product using the duality between homology and cohomology with
compact support.

The ends of $\M$ can be naturally identified with the ends of the tree
$T$. The sets $E_{\Sigma}^\pm$ correspond to the ends of the two
components of $T-\{v_0\}$. It is easy to see that $\widetilde{\Sigma}$
and $\widetilde{S}$ cross if and only if each of the sets
$E_{\Sigma}^\pm$ contain pairs of ends $E_1$ and $E_2$ which are in
different equivalence classes with respect to the partition
corresponding to $\widetilde{S}$. By symmetry, it suffices to consider
the case of $E_{\Sigma}^+$. Let $X$ denote the closure of the
component of $\M-\widetilde\Sigma_j$ with $ends(X)=E_{\Sigma}^+$.

As $v_0$ is an internal vertex of the tree $T(\widetilde S)$, there is
a terminal vertex $w$ of $T(\widetilde S)$ contained in $X$. A
terminal vertex of $T(\widetilde S)$ corresponds to a piece which is a
disc $D$ in a $3$-punctured sphere $P$, with $P$ the closure of a
component of $\widetilde{M}-\widetilde{\Sigma}$. Let $Q_1$ and $Q_2$
denote the boundary components of $P$ disjoint from $D$ (hence from
$S$). Then $D$ separates $Q_1$ and $Q_2$.

For $i=1,2$, let $W_i$ denote the closure of the component of $\M-Q_i$
which does not contain $S$. As $Q_i$ is the lift of an essential
sphere, and $\M$ is simply-connected, $Q_i$ is non-trivial as an
element of $H_2(\M)$. Hence $W_i$ is non-compact. By construction
$W_i\subset X$, hence the ends of $W_i$ are contained in $E_\Sigma^+$.

As $D$ separates $Q_1$ and $Q_2$, (after possibly interchanging $Q_1$
and $Q_2$) there is a path $c:[0,1]\to P$ intersecting $S$
transversely in one point (with the sign of the intersection $+1$) so
that $c(0)\in Q_1$ and $c(1)\in Q_2$. As $W_1$ and $W_2$ are
non-compact, we can extend $c$ to a proper function $c:\R\to \M$ with
$c((-\infty,0))\subset W_1$ and $c((1,\infty))\subset W_2$.

The ends $E_1$ and $E_2$ of $c$ are ends of $X$ (as $W_i\subset X$ for
$i=1,2$). Further, by construction $c\cdot S=1$. It follows that the
ends $E_1,E_2\subset E_\Sigma^+$ are in different components with
respect to the partition corresponding to $S$. By symmetry, we can
find a similar pair of ends in $E_\Sigma^-$. It follows that
$\widetilde{S}$ and $\widetilde{\Sigma}$ cross.

\end{proof}

We now complete the proof of Lemma~\ref{L:ineq}. We have seen that the
number of components of $S\cap \Sigma_j$ is the number of components
of the pre-image of $\Sigma_j$ which intersect $\widetilde{S}$. For a
fixed lift $\widetilde{\Sigma_j}$ of $\Sigma_j$, the components of the
pre-images of $\Sigma_j$ are the translates $g\widetilde{\Sigma_j}$ of
$\widetilde{\Sigma_j}$.

By Lemma~\ref{L:cross}, it follows that if $\widetilde S$ intersects
$g\widetilde{\Sigma_j}$, then $\widetilde{S}$ crosses
$g\widetilde{\Sigma_j}$. The converse of this is obvious. By the
definition of algebraic intersection number, Lemma~\ref{L:ineq}
follows.

\end{proof}

\begin{proof}[Proof of Theorem~\ref{T:main}]

We have seen that it suffices to consider the case when $A=[S]$,
$B=[\Sigma_1]$ and $S$ is in normal form with respect to $\Sigma$. By
Lemma~\ref{L:ineq}, $\widetilde{I}(A,B)\geq I(A, B)$.

Conversely, let $S$ and $\Sigma_1$ be embedded spheres with $A=[S]$,
$B=[\Sigma_1]$ and $I(A,B)=\vert S\cap \Sigma_1\vert$. Let
$\widetilde{S}$ and $\widetilde{\Sigma_1}$ be lifts of $S$ and
$\Sigma_1$, respectively, to $\M$. Observe that (distinct) components
of intersection of $S$ with $\Sigma_1$ lift to (distinct) components
of intersection of $\widetilde{S}$ with translates of
$\widetilde{\Sigma_1}$. Hence the number of translates of
$\widetilde{\Sigma_1}$ that intersect $\widetilde{S}$ is at most
$I(A,B)$. As $\tilde{I}(A,B)$ is the number of translates of
$\widetilde{\Sigma_1}$ that \emph{cross} $\widetilde{S}$, and
components that cross must intersect, it follows that
$\tilde{I}(A,B)\leq I(A,B)$.

This completes the proof of the theorem.

\end{proof}

Our methods also yield the following result. This also follows from
the work of Scott and Swarup, see ~\cite{SS}.

\begin{theorem}\label{T:disj}
 If $A_1$,\dots,$A_n$ are isotopy classes of embedded spheres in $M$
 such that, for $1\leq i,j\leq n$, $A_i$ and $A_j$ can be represented
 by disjoint spheres, then there exist disjointly embedded spheres
 $S_i$, $1\leq i\leq n$, such that $A_i = [S_i]$.
\end{theorem}

\begin{proof}
We prove this by induction on $n$. For $n=1,2$, the conclusion is
immediate from the hypothesis. Assume that the result holds for
$n=k$ and consider a collection $A_i$ as in the hypothesis with
$n=k+1$.

Suppose one of the spheres, which we can assume without loss of
generality is $A_n$, is not essential. By the induction hypothesis,
there are disjoint embedded spheres $S_i$, $1\leq i<n$, with
$[S_i]=A_i$. Choose a $3$-ball disjoint from the spheres $S_i$, $1\leq
i<n$ and let $S_n$ be its boundary. Then the spheres $S_i$, $1\leq
i\leq n$, give the required collection.

Thus we may assume that all the isotopy classes $A_i$ of spheres
are essential. By induction hypothesis, there are disjoint
embedded spheres $S_i$, $1\leq i<n$, with $[S_i]=A_i$. As these
are essential by our assumption, we can extend the collection
$S_i$ to a maximal system of spheres. We let $S_n$ be a sphere in
normal form with respect to this collection. By hypothesis,
$I(S_n,S_i)=0$ for $1\leq i\leq n$. By the proof of
Lemma~\ref{L:ineq}, it follows that $S_n$ is disjoint from $S_i$.
Thus, $S_i$, $1\leq i\leq n$, is a collection of disjoint embedded
spheres with $A_i=[S_i]$.

\end{proof}

\begin{remark}
 The above theorem shows that the sphere complex
associated to $M$ is a full complex in the sense that if $ V_1,
V_2,....,V_k$ are the vertices of the sphere complex and if there
is an edge between every pair $V_i, V_j$ of vertices, where
$1\leq i,j \leq k$, then these vertices bound a simplex in the
sphere complex.
\end{remark}

\bibliographystyle{amsplain}

\end{document}